\documentclass[10pt,a4paper]{amsart}
\usepackage{amsmath,amsthm,amssymb,amsfonts}

\usepackage[pdftex]{color}
\usepackage[bookmarks=true,hyperindex,pdftex,colorlinks,citecolor=blue,
linkcolor=blue,urlcolor=blue]{hyperref}

\usepackage[english]{babel}

\usepackage{graphicx}

\theoremstyle{plain}
\newtheorem{theorem}{Theorem}[section]
\newtheorem{cor}[theorem]{Corollary}
\newtheorem{prop}[theorem]{Proposition}
\newtheorem{lemma}[theorem]{Lemma}

\theoremstyle{definition}

\newtheorem{example}[theorem]{Example}

\newtheorem{question}[theorem]{Question}

\newtheorem{definition}[theorem]{Definition}

\newcommand{\R}{\mathbb{R}}
\newcommand{\N}{\mathbb{N}}
\newcommand{\C}{\mathbb{C}}

\newcommand{\Lin}{\mathcal{L}}
\newcommand{\K}{\mathcal{K}}

\newcommand{\eps}{\varepsilon}
\newcommand{\lam}{\lambda}

\DeclareMathOperator{\dist}{dist}

\DeclareMathOperator{\sgn}{sgn}
\DeclareMathOperator{\re}{Re}

\DeclareMathOperator{\NA}{NA}

\DeclareMathOperator{\Id}{Id}

\DeclareMathOperator{\BS}{B\check{S}}

\renewcommand{\leq}{\leqslant}
\renewcommand{\geq}{\geqslant}

\title{The Birkhoff-James orthogonality and norm attainment for multilinear maps}

\author[G.~Choi]{Geunsu Choi}
\address[G.~Choi]{Department of Mathematics, Institute for Industrial and Applied Mathematics, Chungbuk National University, Cheongju, Chungbuk 28644, Republic of Korea}
\email{\texttt{chlrmstn90@gmail.com}}

\author[S.~K.~Kim]{Sun Kwang Kim}
\address[S.~K.~Kim]{Department of Mathematics, Chungbuk National University, Cheongju, Chungbuk 28644, Republic of Korea}
\email{\texttt{skk@chungbuk.ac.kr}}
\thanks{The first and second authors were supported by the National Research Foundation of Korea(NRF) grant funded by the Korea government(MSIT) [NRF-2020R1C1C1A01012267]. }

\date{\today}

\keywords{Banach space, norm attainment, Birkhoff-James orthogonality, Bhatia-\v{S}emrl property, multilinear map}
\subjclass[2010]{Primary 46B04; Secondary  46B07, 46B20}

\begin{document}
	
\begin{abstract}
Very recently, motivated by the result of Bhatia and \v{S}emrl which characterizes the Birkhoff-James orthogonality of operators on a finite dimensional Hilbert space in terms of norm attaining points, the Bhatia-\v{S}emrl property was introduced. The main purpose of this article is to study the denseness of the set of multilinear maps with the Bhatia-\v{S}emrl property which is contained in the set of norm attaining ones. Contrary to the most of previous results which were shown for operators on real Banach spaces, we prove the denseness for multilinear maps on some complex Banach spaces. We also show that the denseness of operators does not hold when the domain space is $c_0$ for arbitrary range. Moreover, we find plenty of Banach spaces $Y$ such that only the zero operator has the Bhatia-\v{S}emrl property in the space of operators from $c_0$ to $Y$.
\end{abstract}

\maketitle

\section{Introduction}
The concept of orthogonality on a normed space was first considered by G. Birkhoff \cite{B} in 1935, known as the Birkhoff-James orthogonality. This extends the classical orthogonality on a Hilbert space to a general Banach space, and does an important role in figuring out the geometric structure of the space. Later in 1999, R. Bhatia and P. \v{S}emrl \cite{BS} found a way to characterize the orthogonality of operators on a finite dimensional Hilbert space in the sense of the norm attainment of operators. Afterwards it was shown by C. K. Li and H. Schneider \cite{LS} that such characterization is not applicable in general even for operators on finite dimensional spaces. In the  recent decade, there have been many works to study operators that the characterization holds, and we say that such operators have the Bhatia-\v{S}emrl property \cite{PS,PSG,SP,SPH}. Our goal of the present paper is also to study multilinear maps such characterization holds.

For a better understanding, we shall introduce some required basic terminologies and backgrounds here. Unless it is written specifically, we denote Banach spaces by $X$, $X_i$ $(1\leq i \leq N,~N\in \mathbb{N})$ and $Y$ over a base field $\mathbb{K} = \R \text{ or } \C$. We write $B_X$ and $S_X$ for the unit ball and unit sphere of $X$, respectively. We write $\Lin^N(X_1,\ldots,X_N;Y)$ to be the space of all $N$-linear maps from $X_1 \times \cdots \times X_N$ into $Y$ equipped with the typical supremum norm on $S_{X_1} \times \cdots \times S_{X_N}$. Especially, we denote the case when $N=1$, the space of operators from $X$ into $Y$, by $\Lin(X;Y)$ and the topological dual space $\Lin(X;\mathbb{K})$ by $X^*$. An $N$-linear map $A\in \Lin^N(X_1,\ldots,X_N;Y)$ is said to \emph{attain its norm} at $(x_1,\ldots,x_N)  \in S_{X_1} \times \cdots \times S_{X_N}$ if $\|A(x_1,\ldots,x_N)\|=\|A\|$, and we write as $A \in \NA(X_1,\ldots,X_N;Y)$ and $(x_1,\ldots,x_N) \in M_A := \{(z_1,\ldots,z_N) \in S_{X_1} \times \cdots \times S_{X_N} \colon \|A(z_1,\ldots,z_N)\|=\|A\| \}$.  We now introduce the orthogonality on Banach spaces which is the main concept of this article.

\begin{definition} \cite{B}
We say a vector $x \in X$ is \emph{orthogonal to $y \in X$ in the sense of Birkhoff-James} if $\|x\| \leq \|x+ \lam y \|$ for any $\lam \in \mathbb{K}$, and it is denoted by $x \perp_B y$.
\end{definition}

Note that the Birkhoff-James orthogonality is not a symmetric notion. A useful characterization of the orthogonality on $X$ is given as follows \cite{J}:
$$\text{For~} x,y \in X, x \perp_B y \text{~ if and only if there exists~}x^* \in S_{X^*} \text{~such that~}x^*(x)=\|x\| \text{~and~}x^*(y)=0.$$

 If we take operators $T \in \NA(X;Y)$, $S \in \Lin(X;Y)$ and a point $x_0 \in M_T$, one may easily see that $Tx_0 \perp_B Sx_0$ implies that $T \perp_B S$. As it is mentioned in the introduction, Bhatia and \v{S}emrl showed in \cite{BS} that the opposite direction of implication still holds when $X=Y$ is a finite dimensional complex Hilbert space. Motivated by this result, the Bhatia-\v{S}emrl property is defined as follows.

\begin{definition} \cite{SPH}
A bounded linear operator $T \in \Lin(X;Y)$ is said to have the \emph{Bhatia-\v{S}emrl property} (in short, \emph{$\BS$ property}) if for any $S \in \Lin(X;Y)$ with $T \perp_B S$, there exists $x_0 \in S_X$ such that $\|Tx_0\|=\|T\|$ and $Tx_0 \perp_B Sx_0$. The set of operators with the $\BS$ property is denoted by $\BS(X;Y)$.
\end{definition} 

It is easy to see that every operator in $\Lin(X;\mathbb{K})$ has the $\BS$ property if and only if $X$ is reflexive from the characterization of Birkhoff-James orthogonality given by James. However, it is known that there is an operator without the $\BS$ property in general \cite{LS}, and so many authors are interested in investigating the ``quantity" of  operators with the property. More precisely, in \cite{SPH}, it is shown that the set of operators with the $\BS$ property defined on a real finite dimensional strictly convex Banach space $X$ is always dense in $\Lin(X;X)$. After that Paul, Sain and Ghosh found a useful theorem in \cite{PSG}, and we introduce here for its importance.

\begin{theorem}\cite[Theorem 2.2]{PSG} \label{theorem:real-BS}
Let $X$ and $Y$ be \textbf{real} Banach spaces, and $T \in \Lin(X;Y)$ be such that $M_T = C \cup -C$ for some nonempty compact connected $C \subset X$. Suppose that $\sup_{x \in D} \|Tx\| < \|T\|$ whenever $D$ is a closed subset of $S_X$ with $\dist(M_T,D)>0$. Then, for any $S \in \Lin(X;Y)$, $T \perp_B S$ if and only if there exists $x_0 \in M_T$ such that $Tx_0 \perp_B Sx_0$.
\end{theorem}

With the aid of the above result, one may find a way to measure the denseness of operators with the $\BS$ property in many cases. Kim proved that some conditions on pairs of Banach spaces ensure the denseness of operators with the $\BS$ property such as when $X$ has the Radon-Nikod\'ym property (in short, RNP) \cite{K}. This is a generalization of the result in \cite{SPH} since it is known that every finite dimensional space has the RNP. Later, Kim and Lee \cite{KL} complemented the result  with the case when $X$ has so-called property $\alpha$ or when $\BS(X;\R)$ is dense in $\Lin(X;\R)$ and $Y$ has property $\beta$.

The main goal of Section 2 is to generalize and strengthen aforementioned results in \cite{K,KL} for multilinear maps. Similarly to the case of operators, it is clear that $A \in \NA(X_1,\ldots,X_N;Y)$ and $B \in \Lin^N(X_1,\ldots,X_N;Y)$ satisfy $A\perp_B B$ if $A(x_1,\ldots,x_N)\perp_B B(x_1,\ldots,x_N)$ for some $(x_1,\ldots,x_N)\in M_A$. What we are going to see is whether the converse holds or not.

\begin{definition}
An $N$-linear map $A \in \Lin^N(X_1,\ldots,X_N;Y)$ is said to have the \emph{$\BS$ property} if for any $B \in \Lin^N(X_1,\ldots,X_N;Y)$ with $A \perp_B B$, there exists $(x_1,\ldots,x_N) \in M_A$ such that  $A(x_1,\ldots,x_N) \perp_B B(x_1,\ldots,x_N)$, and we write by $A \in \BS(X_1,\ldots,X_N;Y)$ in this case.
\end{definition}


 We observe first that an $N$-linear map $A \in \Lin^N(X_1,\ldots,X_N;Y)$ can be identified by an operator $T_A \in \Lin(X_1; \Lin^{N-1}(X_2,\ldots,X_N;Y))$ with the canonical isometric relation $A(x_1,\ldots,x_N) = (T_Ax_1)(x_2,\ldots,x_N)$, and so we have that $A,B \in \Lin^N(X_1,\ldots,X_N;Y)$ satisfy $A\perp_B B$ if and only if $T_A\perp_B T_B$. From this identification, for a finite dimensional complex Hilbert space $H$, we rewrite the result of Bhatia and \v{S}emrl in terms of bilinear forms as follows. 
$$\text{For every~}A,B \in \Lin^2(H,H;\mathbb{C}),~A\perp_B B \text{~if and only if~} A(x_1,x_2)\perp_B B(x_1,x_2) \text{~for some~}(x_1,x_2)\in M_A.$$

Indeed, the result of Bhatia and \v{S}emrl \cite{BS} says that $T_Ax_1\perp_B T_Bx_1$  for some $x_1\in M_{T_A}$. Since Hilbert spaces are isometrically isomorphic to their dual spaces, there exists $x_2 \in S_H (=S_{H^{**}})$ such that $x_2(T_Ax_1)=\|T_A x_1\|$ and $x_2(T_Bx_1)=0$ by the characterization of James \cite{J}. This shows that $(x_1,x_2)\in M_A$ and $A(x_1,x_2)\perp_B B(x_1,x_2)$. In this argument, in fact, we used only the facts that every operator $T\in \mathcal{L}(H,H^*)$ has the $\BS$ property and that $H^*$ is reflexive. Hence, we deduce the following.

\begin{lemma}\label{easylemma} For Banach spaces $X$ and $Y$, if $Y$ is reflexive, then
$A \in \Lin^2(X,Y;\mathbb{K})$ has the $\BS$ property if and only if the corresponding operator $T_A \in \Lin(X;Y^*)$ has the $\BS$ property.
\end{lemma}

 As a consequence of the equivalence, we obtain a positive result by applying \cite[Corollary 3.5]{KL} with the fact that $\ell_\infty^n$ is a reflexive space having property $\beta$.
\begin{cor}
Let $X$ be a locally uniformly convex Banach space and $n \in \N$. Then, $\BS(X,\ell_1^n;\mathbb{\mathbb{K}})$ is dense in $\Lin^2(X,\ell_1^n;\mathbb{K})$.
\end{cor}

In the above statement, it should be written by $\mathbb{R}$ instead of $\mathbb{K}$ since \cite[Corollary 3.5]{KL} is proved for real spaces. However, in Section 2 we prove a stronger statement for complex spaces, and this is why we put $\mathbb{K}$. We also comment that the version of Lemma \ref{easylemma} for a non-reflexive $Y$ does not hold. Since $\Lin^2(\mathbb{K},Y;\mathbb{K})$ can be identified with $\mathcal{L}(Y;\mathbb{K})$ and there exists an operator in $\mathcal{L}(Y;\mathbb{K})$ without the $\BS$ property, we can find $A \in \Lin^2(\mathbb{K},Y;\mathbb{K})$ without the $\BS$ property whenever $Y$ is non-reflexive. However, it is clear that the corresponding operator $T_A$ belongs to $\BS(\mathbb{K}; Y^*)$.

We remark that previous results in \cite{K,KL} are only valid for real cases, as the construction of $C$ in Theorem \ref{theorem:real-BS} strongly depends on the disconnectedness of two specific partitions. We present a similar result to Theorem \ref{theorem:real-BS} in Section \ref{section:denseness} which is slightly weaker but which also covers the complex case, and deduce some denseness results.

On the opposite hand, there are still many examples which shows that $\BS(X;Y)$ is a very small set compared to $\Lin(X;Y)$. As shown in \cite{K}, there is no nontrivial operator with the $\BS$ property in $\Lin(c_0;Y)$ when $Y$ is strictly convex and moreover that $\BS(c_0;c_0)$ is not dense in $\Lin(c_0;c_0)$. In \cite{KL}, the authors proved that $\BS(L_1[0,1];Y)=\{0\}$ for an arbitrary Banach space $Y$. It is a very intriguing result as there are many range spaces $Y$ such that $\NA(L_1[0,1];Y)$ is dense in $\Lin(L_1[0,1];Y)$ on the contrary. In Section \ref{section:c0}, we strengthen the result on $c_0$. More precisely, we show first that there are plenty of Banach spaces $Y$ such that $\BS(c_0;Y)$ only consists of the zero operator, and secondly a quite stunning result that $\BS(c_0;Y)$ is never dense in $\Lin(c_0;Y)$ for every Banach space $Y$. This eventually gives rise to producing a negative result on the denseness of multilinear maps when one of the domain space is $c_0$.

\section{Denseness of $N$-linear maps with the Bhatia-\v{S}emrl property}\label{section:denseness}

We provide in this section many positive examples of tuples $(X_1,X_2,\ldots,X_N,Y)$ of Banach spaces such that the set of $N$-linear maps with the $\BS$ property is dense in $\Lin^N(X_1,\ldots,X_N;Y)$. As many previous results for operators rely on Theorem \ref{theorem:real-BS}, it is inevitable to compromise Theorem \ref{theorem:real-BS} also to cover the complex case, and a stronger assumption will help to handle those situations. In the followings, $\mathbb{T}^N$ denotes the $N$ product of unit spheres of the scalar field $\mathbb{K}$.


\begin{prop}\label{prop:one-point-BS}
Let $X_1, \ldots, X_N$ and $Y$ be Banach spaces. Let $A \in \Lin^N(X_1, \ldots, X_N;Y)$ be such that $M_A =  \{( \theta^1 x_0^1, \ldots, \theta^N x_0^N) \in S_{X_1} \times \cdots \times S_{X_N} \colon (\theta^1, \ldots, \theta^N) \in \mathbb{T}^N \}$ for some $(x_0^1, \ldots, x_0^N) \in S_{X_1} \times \cdots \times S_{X_N}$. Suppose that $\sup_{(x^1,\ldots,x^N) \in D} \|A(x^1, \ldots, x^N)\| < \|A\|$ whenever $D$ is a closed subset of $S_{X_1} \times \cdots \times S_{X_N}$ with $\dist(M_A,D)>0$. Then, for any $B \in \Lin^N(X_1, \dots, X_N;Y)$, $A \perp_B B$ if and only if $A(x_0^1,\ldots,x_0^N) \perp_B B(x_0^1, \ldots, x_0^N)$.
\end{prop}

\begin{proof}
Since the `if' part is evident, suppose that $A \perp_B B$ and there exists $\lam_0 \in \mathbb{K}$ such that $\|A(x_0^1,\ldots,x_0^N)\| > \|A(x_0^1,\ldots,x_0^N) + \lam_0 B (x_0^1,\ldots,x_0^N) \|$. By assumption, we can choose a sequence $\{(x_j^1,\ldots,x_j^N)\}_{j=1}^\infty \subset S_{X_1} \times \cdots \times S_{X_N}$ so that
$$
\left\| A(x_j^1, \ldots, x_j^N) + \frac{\lam_0}{j} S(x_j^1,\ldots,x_j^N) \right\| \geq \|A\| - \frac{1}{j^2} \qquad \text{for each } j \in \N.
$$
For $D_n := \overline{\{(x_j^1,\ldots,x_j^N)\}_{j=n}^{\,\infty}} \subset S_{X_1} \times \cdots \times S_{X_N}$, if there exist only finite number of indices $j_1< \cdots< j_k$ such that $(\theta_i^1x_{j_i}^1, \ldots,\theta_i^N x_{j_i}^N) = (x_0^1,\ldots,x_0^N)$ for some $(\theta_i^1,\ldots,\theta_i^N) \in \mathbb{T}^N$ for each $1 \leq i \leq k$ then we take $m=j_k+1$, and $m=1$ otherwise.  We now show that $\dist(M_A,D_m)>0$. Otherwise, there is a subsequence $\{(x_{\sigma(j)}^1,\ldots,x_{\sigma(j)}^N)\}$ such that $(x_{\sigma(j)}^1,\ldots,x_{\sigma(j)}^N)$ converges to $(\theta_0^1 x_0^1,\ldots,\theta_0^N x_0^N)$ for some $(\theta_0^1,\ldots,\theta_0^N) \in \mathbb{T}^N$. It follows that
\begin{align*}
\|A\| &= \|A(x_0^1,\ldots,x_0^N)\| \\
&> \|A(x_0^1,\ldots,x_0^N) + \lam_0 B(x_0^1,\ldots,x_0^N)\| \\
& = \lim_{j \to \infty} \left\| A(x_{\sigma(j)}^1,\ldots,x_{\sigma(j)}^N) + \lam_0 B(x_{\sigma(j)}^1,\ldots,x_{\sigma(j)}^N) \right\| \\
& \geq \lim_{j \to \infty} \left( \sigma(j) \left\|A(x_{\sigma(j)}^1,\ldots,x_{\sigma(j)}^N) + \frac{\lam_0}{\sigma(j)} B(x_{\sigma(j)}^1,\ldots,x_{\sigma(j)}^N) \right\| - (\sigma(j)-1) \|A(x_{\sigma(j)}^1,\ldots,x_{\sigma(j)}^N)\| \right) \\
& \geq \lim_{j \to \infty} \left( \sigma(n) \left( \|A\| - \frac{1}{\sigma(j)^2}\right) - (\sigma(j)-1) \|A\| \right) \\
&= \lim_{j \to \infty} \left( \|A\| - \frac{1}{\sigma(j)} \right) \\
& = \|A\|,
\end{align*}
which is a contradiction.
 
Hence, we get
$$
\eps := \|A\| -\sup_{j\geq m} \|A(x_j^1,\ldots,x_j^N)\| > 0.
$$
Take $M \in \N$ so that $M > \max \biggl\{m, \dfrac{2}{\eps}|\lam_0| \|B\| , \sqrt{\dfrac{2}{\eps}} \biggr\}$, then we have the following desired contradiction
\begin{align*}
\|A\|-\frac{\eps}{2} < \|A\| - \frac{1}{M^2} &\leq \|A(x_M^1,\ldots,x_M^N) + \frac{\lam_0}{M} B(x_M^1,\ldots,x_M^N)\| \\
&\leq \|A(x_M^1,\ldots,x_M^N)\| + \frac{|\lam_0|}{M} \|B(x_M^1,\ldots,x_M^N)\| \\
&\leq (\|A\|-\eps) + \frac{\eps}{2} \\
&= \|A\| -\frac{\eps}{2}.
\end{align*}
\end{proof}

We now recall some concepts, namely property quasi-$\alpha$, as a sufficient condition on the domain spaces for the denseness of $N$-linear maps with the $\BS$ property. This is introduced in \cite{CS} as a sufficient condition for the  denseness of norm attaining operators. We note that property quasi-$\alpha$ is stricrtly weaker than property $\alpha$ and it is shown in \cite{KL} that $\BS(X;Y)$ is dense in $\Lin(X;Y)$ whenever $X$ has property $\alpha$ for real Banach spaces.

\begin{definition}
A Banach space $X$ is said to have \emph{property quasi-$\alpha$} if there exist an index set $I$, $\{x_\alpha\}_{\alpha \in I} \subset S_X$, $\{x_\alpha^*\}_{\alpha \in I} \subset S_{X^*}$, and $\lambda \colon \{x_\alpha\}_{\alpha \in I} \to \R$ satisfying that
\begin{enumerate}
\item[\textup{(i)}] $x_\alpha^*(x_\alpha)=1$ for all $\alpha \in I$,
\item[\textup{(ii)}] $|x_\alpha^*(x_\beta)| \leq \lambda(x_\alpha) < 1$ for all $\alpha, \beta \in I$ with $\alpha \neq \beta$,
\item[\textup{(iii)}] For every $e \in \operatorname{Ext}(B_{X^{**}})$, there exists $I_e \subset I$ such that $te \in \overline{\{x_\alpha\}}^{w^*}_{\alpha \in I_e}$ for some $t \in \mathbb{T}$ and $r_e = \sup \{ \lambda(x_\alpha) \colon \alpha \in I_e \} <1$.
\end{enumerate}
\end{definition}

Motivated by the proof of \cite[Theorem 2.1]{CS}, we show the following proposition.

\begin{prop}\label{alpha}
Let $X,X_1, \ldots, X_N$ and $Y$ be Banach spaces. If $X$ has property quasi-$\alpha$ and $\BS(X_1,\ldots,X_N;Y)$ is dense in $\Lin^N(X_1,\ldots,X_N;Y)$, then $\BS(X,X_1,\ldots,X_N;Y)$ is dense in $\Lin^{N+1}(X,X_1,\ldots,X_N;Y)$.
\end{prop} 

\begin{proof} Before proving the statement, we first see the canonical isometry shows that 
$$\Lin^{N+1}(X,X_1,\ldots,X_N;Y)=\Lin(X; \Lin^N(X_1,\ldots,X_N;Y)).$$

For convenience, we set $Z=\Lin^N(X_1,\ldots,X_N;Y)$ and $T_D\in \Lin(X;Z)$ to be the image of $D\in \Lin^{N+1}(X,X_1,\ldots,X_N;Y)$ by the isometry. 

We show that it is possible to approximate $A\in \Lin^{N+1}(X,X_1,\ldots,X_N;Y)$ by $B\in \BS(X,X_1,\ldots,X_N;Y)$. We here assume that $\|A\|=1$ without any loss of generality, and assume that $(T_A)^{**} \in \NA(X^{**};Z^{**})$ by Lindenstrauss (see \cite[Theorem 2]{L}). Let $\{x_\alpha\}_{\alpha \in I} \subset S_{X}$ and $\{x_\alpha^*\}_{\alpha \in I} \subset S_{X^*}$ be as in the definition of property quasi-$\alpha$. Since $(T_A)^{**}$ attains its norm at some $e \in \operatorname{Ext}(B_{X^{**}})$ (see \cite[Theorem 5.8]{Lima}), there exists an index set $I_e \subset I$ as in the definition. 

For arbitrary $\eps>0$, take $0<\delta<\dfrac{\eps}{2}$ so that
$$
1+r_{e}\left( \frac{\eps}{2} + \delta \right) <\left( 1 + \frac{\eps}{2} \right) (1-\delta) .
$$
As it is known that $te \in \overline{\{x_\alpha\}}^{w^*}_{\alpha \in I_e}$ for some $t \in \mathbb{T}$, we can choose $\alpha_0 \in I_e$ such that
$$
\|T_Ax_{\alpha_0}\|>\|T_A\|-\delta= 1-\delta.
$$
By the assumption that $\BS(X_1,\ldots,X_N;Y)$ is dense in $Z$, there exists $U \in \BS(X_1,\ldots,X_N;Y)$ with $\|U\|=\|T_Ax_{\alpha_0}\|$ such that $\|U-T_Ax_{\alpha_0}\|<\delta$. 

The operator $S \in \Lin(X;Z)$ defined by
$$
S(\cdot) := T_A(\cdot) + \left[ \left( 1+\frac{\eps}{2} \right) U - (T_Ax_{\alpha_0}) \right] x^*_{\alpha_0}(\cdot)
$$
 attains its norm at $x_{\alpha_0}$. Indeed, we have that $Sx_{\alpha_0} = \left(1 +\dfrac{\eps}{2}\right) U \in \BS(X_1,\ldots,X_N;Y)$ and so 
$$
\left\|Sx_{\alpha_0}\right\|> \left(1 +\frac{\eps}{2}\right)(1-\delta).
$$
On the other hand, for $\alpha \in I \setminus \{\alpha_0\}$,
\begin{align*}
 \|Sx_{\alpha}\|
&\leq  \left\|T_Ax_{\alpha}\right\| + \left(  \left\|\frac{\eps}{2}\, U\right\|+\left\|U- T_Ax_{\alpha_0} \right\| \right) \left|x_{\alpha_0}^*(x_{\alpha})\right|\\
&\leq 1 + r_e \left( \frac{\eps}{2} + \delta \right)<\left(1 +\frac{\eps}{2}\right)(1-\delta).
\end{align*}

Let $B\in \Lin^{N+1}(X,X_1,\ldots,X_N;Y)$ be the corresponding $(N+1)$-linear map of $S$, and it is enough to prove that $B$ has the $\BS$ property. To do so, we first check the conditions of Proposition \ref{prop:one-point-BS} to see $S \in \BS(X;Z)$. Indeed, for $0<\gamma<2$ and $x \in S_X$ such that $\dist \left(\{\theta x_{\alpha_0} : \theta \in \mathbb{T}\}, x\right) > \gamma$, we show that 
$$\|Sx\| \leq \|S\| - \frac{\gamma\left(\|S\|-\left(1 +\dfrac{\eps}{2}\right)(1-\delta)\right)}{4}.$$

From the fact that the absolutely closed convex hull of $\{x_\alpha\}_{\alpha \in I}$ is $B_X$ which can be deduced by the usual separation argument, choose $n \in \N$, an absolutely convex series $\{c_{\alpha_i}\}_{i=0}^n \subset B_\mathbb{K}$ (indeed, $\sum_{i=0}^n|c_{\alpha_i}|\leq1$) and distinct elements $x_{\alpha_1}, \ldots, x_{\alpha_n} \in \{x_\alpha\}_{\alpha \in I}$ such that $z = \sum_{i=0}^n c_{\alpha_i} x_{\alpha_i}$ satisfies
$$
\|z-x\| < \frac{\gamma\left(\|S\|-\left(1 +\dfrac{\eps}{2}\right)(1-\delta)\right)}{4\|S\|} \quad \text{and} \quad \dist\left(\{\theta x_{\alpha_0} : \theta \in \mathbb{T}\},z\right)>\gamma.
$$
If $|c_{\alpha_0}| > 1- \gamma/2$, then
\begin{align*}
\left\|\frac{c_{\alpha_0}}{|c_{\alpha_0} |} x_{\alpha_0} - z \right\|
&\leq \left\|\frac{c_{\alpha_0}}{|c_{\alpha_0} |} x_{\alpha_0} -c_{\alpha_0} x_{\alpha_0} \right\|+\left\|c_{\alpha_0} x_{\alpha_0}-z \right\| \\
&< \frac{\gamma}{2}+\left\|\sum_{i=1}^n c_{\alpha_i} x_{\alpha_i} \right\|\\
&< \gamma.
\end{align*}
Hence we have $|c_{\alpha_0}| \leq 1 - \gamma/2$, and this leads to that
\begin{align*}
\|Sz\| &\leq |c_{\alpha_0}| \|Sx_{\alpha_0}\| + \sum_{i=1}^n |c_j| \|Sx_{\alpha_j}\| \\
&\leq |c_{\alpha_0}| \|S\| + ( 1- |c_{\alpha_0}|) \left(1 +\frac{\eps}{2}\right)(1-\delta) \\
&\leq \|S\| - (1-|c_{\alpha_0}|)\left(\|S\| -\left(1 +\frac{\eps}{2}\right)(1-\delta)\right) \\
&\leq \|S\| - \frac{\gamma\left(\|S\|-\left(1 +\dfrac{\eps}{2}\right)(1-\delta)\right)}{2}.
\end{align*}
Thus, it follows that
$$
\|Sx\|\leq \|Sz\|+\|Sx-Sz\| \leq \|Sz\|+\|S\|\|x-z\| \leq \|S\| - \frac{\gamma\left(\|S\|-\left(1 +\dfrac{\eps}{2}\right)(1-\delta)\right)}{4}.
$$

Therefore, whenever $B\perp_B C$ for $C\in \Lin^{N+1}(X,X_1,\ldots,X_N;Y)$ we have that $Sx_{\alpha_0}\perp_B T_Cx_{\alpha_0}$. Since we know that $Sx_{\alpha_0}=\left(1+\dfrac{\eps}{2}\right)U$ has the $\BS$ property, there exists $(x_{\beta_1},\ldots,x_{\beta_N})\in M_{Sx_{\alpha_0}}$ so that 
$$\|Sx_{\alpha_0}(x_{\beta_1},\ldots,x_{\beta_N}) \|=\|Sx_{\alpha_0}\| \text{~and~} Sx_{\alpha_0}(x_{\beta_1},\ldots,x_{\beta_N}) \perp_B T_Cx_{\alpha_0}(x_{\beta_1},\ldots,x_{\beta_N}).$$
From the facts that $Sx_{\alpha_0}(x_{\beta_1},\ldots,x_{\beta_N}) =B(x_{\alpha_0},x_{\beta_1},\ldots,x_{\beta_N})$ and $T_Cx_{\alpha_0}(x_{\beta_1},\ldots,x_{\beta_N}) =C(x_{\alpha_0},x_{\beta_1},\ldots,x_{\beta_N})$, we finish the proof.
\end{proof}

 The Radon-Nikod\'ym property, RNP in short, had been considered as an important concept to understand the geometry of infinite dimensional Banach spaces. Especially, Stegall \cite{S} proved the following optimization principle on spaces with the RNP. For a subset $D\subset X$, a real valued function $f$ on $D$ is said to \emph{strongly expose} $D$ if there exists $x_0 \in D$ such that $f(x_0)=\sup_{t\in D}f(t)$ and that every sequence $\{x_n\}_{n=1}^\infty \subset D$ converges to $x_0$ whenever $f(x_n)$ converges to $f(x_0)$.

\begin{lemma}\cite[Stegall's optimization principle]{S}\label{lemma:stegall}
Let $X$ be a Banach space, $D \subset Y$ be a bounded RNP set and $f: D \to \R$ is an upper semi-continuous bounded above function. Then,
$$
0\in \overline{\{x^* \in X^* \colon f + \re x^* \text{ strongly exposes } D\}}
$$
\end{lemma}
For more information on the RNP, we refer the reader to the classical monograph \cite{DU}. In \cite{K}, it is proved for the real case that if $X$ is a Banach space with the RNP, then $\BS(X;Y)$ is dense in $\Lin(X;Y)$ for every Banach space $Y$. Applying the idea of \cite{AFW}, we extend this result to $N$-linear maps, and we cover complex case as well. 

\begin{prop}\label{RNPresult}
Let $X_1,\ldots,X_N$ be Banach spaces with the RNP. Then, $\BS(X_1,\ldots,X_N;Y)$ is dense in $\Lin^N(X_1,\ldots,X_N;Y)$ for every Banach space $Y$.
\end{prop}

\begin{proof}
For an arbitrary $A \in \Lin^N(X_1,\ldots,X_N;Y)$, our goal is to approximate $A$ with elements in $\BS (X_1,\ldots,X_N;Y)$. Without loss of generality, we may assume that $\|A\|=1$. Fix any $0<\eps<1/2$ and we first follow  the idea in \cite{AFW} to find  $B \in \BS (X_1,\ldots,X_N;Y)$ so that $\|B-A\|<\eps$.  Indeed, as $X_1, \ldots,X_{N}$ have the RNP, we see that the finite $\ell_\infty$-sum $X_1 \oplus_\infty \cdots \oplus_\infty X_N$ has the RNP. Thus by Lemma \ref{lemma:stegall}, there exists $\phi = (x_1^*,\ldots,x_N^*) \in X_1^* \oplus_1 \cdots \oplus_1 X_N^*$ with $\|\phi\|<\eps$ such that $\|A(\cdot)\|+\re \phi(\cdot)$ strongly exposes $B_{X_1} \times \cdots \times B_{X_N}$ at some $(x_1,\ldots,x_N)$.

We see that all the elements $x_1,\ldots,x_N$ are nonzero from the inequality
$$
\|A(x_1,\ldots,x_N)\| + \re \left[\phi(x_1,\ldots,x_N)\right] \geq \|A(z_1,\ldots,z_N)\| + \re \left[\phi(z_1,\ldots,z_N)\right]
$$
for all $(z_1,\ldots,z_N) \in B_{X_1} \times \cdots \times B_{X_N}$, since $\|A(x_1,\ldots,x_N)\| + \re \left[\phi(x_1,\ldots,x_N)\right] < \eps < 1/2$ otherwise.

 Thus we see that $\|x_1\|=\ldots=\|x_N\|=1$ by normalization, and so there exists $(w_1^*,\ldots,w_N^*) \in S_{{X_1}^*} \times \cdots \times S_{{X_N}^*}$ such that $w_1^*(x_1)=\ldots=w_N^*(x_N)=1$.
We also note that $|\phi(x_1,\ldots,x_N)|=\phi(x_1,\ldots,x_N)$ and 
$$
\|A(x_1,\ldots,x_N)\| + \re \left[\phi(x_1,\ldots,x_N)\right] \geq \|A(z_1,\ldots,z_N)\| + \left|\phi(z_1,\ldots,z_N)\right|
$$
for all $(z_1,\ldots,z_N) \in B_{X_1} \times \cdots \times B_{X_N}$ by rotation of elements.

 As in \cite{AFW}, we clearly see a map $B \in \Lin^N(X_1,\ldots,X_N;Y)$ given by
$$
B(z_1,\ldots,z_N) := A(z_1,\ldots,z_N) + \sum_{j=1}^n\left(x_j^*(z_j) \prod_{\substack{i=1 \\ i\neq j}}^n w_i^*(z_i)\right)\frac{A(x_1,\ldots,x_N)}{\|A(x_1,\ldots,x_N)\|} \quad \text{for } (z_1,\ldots,z_N) \in X_1 \times \cdots \times X_N
$$
 attains its norm at $(x_1,\ldots,x_N)$, and especially we deduce 
$$\|B\|=\|B(x_1,\ldots,x_N)\|=\|A(x_1,\ldots,x_N)\|+\re \left[\phi(x_1,\ldots,x_N)\right].$$

 We now show that $B$ is the desired one by checking the conditions in Proposition \ref{prop:one-point-BS}. Let
$$M:=\{( \theta^1 x_1, \ldots, \theta^N x_N) \in S_{X_1} \times \cdots \times S_{X_N} \colon (\theta^1, \ldots, \theta^N) \in \mathbb{T}^N \}.$$
If there exists a closed subset $D$ of $S_{X_1} \times \cdots \times S_{X_N}$ so that 
$$\sup_{(z^1,\ldots,z^N) \in D} \|B(z^1,\ldots,z^N)\| = \|B\| \quad \text{and} \quad \dist(M,D)>0,$$
there is a sequence $\left\{(z_j^1,\ldots,z_j^N)\right\}_{j=1}^\infty \subset D$ satisfying $\|B(z_j^1,\ldots,z_j^N)\|$ converges to $\|B\|$. Since
\begin{align*}
\|B(\theta_j z_j^1,\ldots,\theta_j z_j^N)\|&\leq \|A(z_j^1,\ldots,z_j^N)\|+\left|\phi (\theta_j z_j^1,\ldots,\theta_j z_j^N)\right| \\
&= \|A(z_j^1,\ldots,z_j^N)\|+\re \left[\phi (z_j^1,\ldots,z_j^N)\right] \\
&\leq \|A(x_1,\ldots,x_N)\| + \re \left[\phi(x_1,\ldots,x_N)\right]
\end{align*}
for a suitable $\theta_j \in \mathbb{T}$ for each $j$, we have that $(\theta_j z_j^1,\ldots,\theta_j z_j^N)$ converges to $(x_1,\ldots,x_N)$ from the strong exposedness of $\|A(\cdot)\|+\re \phi(\cdot)$, and this contradicts to $\dist(M,D)>0$.
\end{proof}

As a consequence of Propositions \ref{alpha} and \ref{RNPresult}, we have the following.
\begin{cor}\label{alphaRNP}
Let $X_1,\ldots,X_N$ be Banach spaces having at least one of properties among the RNP and property quasi-$\alpha$. Then, $\BS(X_1,\ldots,X_N;Y)$ is dense in $\Lin^N(X_1,\ldots,X_N;Y)$ for every Banach space $Y$.
\end{cor}
\begin{proof} We first consider the case that there are at least two types of different spaces in $X_i$'s such that some has the RNP and the rest has property quasi-$\alpha$. By a suitable rearrangement of spaces, we assume that there exists $1\leq k \leq N-1$ so that $X_i$ has property quasi-$\alpha$ if $1\leq i\leq k$ and it has the RNP otherwise. From Proposition \ref{RNPresult}, we have that $\BS(X_{k+1},\ldots,X_N;Y)$ is dense in $\Lin^{N-k}(X_{k+1},\ldots,X_N;Y)$. Thus Proposition \ref{alpha} shows that $\BS(X_{k},\ldots,X_N;Y)$ is dense in $\Lin^{N-k+1}(X_{k},\ldots,X_N;Y)$. From the usual induction argument we prove the statement.

 Since the case that all of $X_i$'s having the RNP is proved in Proposition \ref{RNPresult}, it remains to show the case when all $X_i$'s having property quasi-$\alpha$. We here take the isometry $Y=\Lin(\mathbb{K};Y)$ and use $N+1$ and $X_{N+1}=\mathbb{K}$ instead of $N$. Since $\mathbb{K}$ has the RNP, it is possible to apply the argument above and we see that $\BS(X_1,\ldots,X_{N+1};Y)$ is dense in $\Lin^{N+1}(X_1,\ldots,X_{N+1};Y)$. Hence, from the canonical isometry, we have that $\BS(X_1,\ldots,X_{N};\Lin(X_{N+1}^*;Y))$ is dense in $\Lin^N(X_1,\ldots,X_{N};\Lin(X_{N+1}^*;Y))$ and so we finish the proof.
\end{proof}

We now move on to a condition for range spaces for the denseness, which is a dual notion of property quasi-$\alpha$.

\begin{definition} \cite{AAP}
A Banach space $X$ is said to have \emph{property quasi-$\beta$} if there exist an index set $I$, $\{x_\alpha\}_{\alpha \in I} \subset S_X$, $\{x_\alpha^*\}_{\alpha \in I} \subset S_{X^*}$, and $\lambda \colon \{x_\alpha^*\}_{\alpha \in I} \to \R$ satisfying that
\begin{enumerate}
\item[\textup{(i)}] $x_\alpha^*(x_\alpha)=1$ for all $\alpha \in I$,
\item[\textup{(ii)}] $|x_\beta^*(x_\alpha)| \leq \lambda(x_\alpha^*) < 1$ for all $\alpha, \beta \in I$ with $\alpha \neq \beta$,
\item[\textup{(iii)}] For every $e^* \in \operatorname{Ext}(B_{X^*})$, there exists $I_{e^*} \subset I$ such that $te^* \in \overline{\{x_\alpha^*\}}^{w^*}_{\alpha \in I_e}$ for some $t \in \mathbb{T}$ and $r_{e^*} = \sup \{ \lambda(x_\alpha^*) \colon \alpha \in I_{e^*} \} <1$.
\end{enumerate}
\end{definition}

Similarly to the case of property quasi-$\alpha$, property quasi-$\beta$ is introduced as a sufficient condition on the range space for the denseness of norm attaining operators. In \cite{KL}, the authors showed that property $\beta$ of a real Banach space $Y$ is a universal condition for $\BS(X;Y)$ being dense in $\Lin(X;Y)$ provided that $\BS(X;\R)$ is dense in $\Lin(X;\R)$. We improve this result with a strictly weaker property, and we refer to \cite{AAP} for more information on property quasi-$\beta$.

\begin{prop}
Let $X_1, \ldots, X_N$ be Banach spaces such that $\BS(X_1,\ldots,X_N;\mathbb{K})$ is dense in $\Lin^N(X_1,\ldots,X_N;\mathbb{K})$. Let $Y$ be a Banach space with property quasi-$\beta$. Then, $\BS(X_1,\ldots,X_N;Y)$ is dense in $\Lin^N(X_1,\ldots,X_N;Y)$.
\end{prop}

\begin{proof}The beginning of the proof is from \cite[Theorem 2]{AAP} and \cite[Theorem 2.12]{CS}, but we give details for the completeness. 

Let $\{x_\alpha\}_{\alpha \in I} \subset S_X$ and $\{x_\alpha^*\}_{\alpha \in I} \subset S_{X^*}$ be the corresponding index sets in the definition of property quasi-$\beta$. By \cite[Proposition 4]{Z}, it suffices to show that every $A \in \Lin^N(X_1,\ldots,X_N;Y)$ with $\|A\|=1$ satisfying ${\widetilde{A}}^* \in \NA(Y^*;Z^*)$ can be approximated by $B \in \BS(X_1,\ldots,X_N;Y)$, where $Z$ is the completed projective tensor product space $X_1 \widehat{\otimes}_\pi \cdots \widehat{\otimes}_\pi X_N$ and $\widetilde{A}$ is a linearization of $A$ on $Z$. As in Proposition \ref{alpha}, we see that $\widetilde{A}^*$ attains its norm at some $e^* \in \operatorname{Ext}(B_{Z^*})$ by \cite[Theorem 5.8]{Lima} and choose an index set $I_{e^*} \subset I$ in the definition of property quasi-$\beta$.

For an arbitrary $\eps>0$, take $0<\delta<\dfrac{\eps}{2}$ so that
$$
1+r_{e^*}\left( \frac{\eps}{2} + \delta \right) < (1-\delta) \left( 1 + \frac{\eps}{2} \right).
$$
As we know that $te^* \in \overline{\{x_\alpha^*\}}^{w^*}_{\alpha \in I}$ for some $t \in \mathbb{T}$, there exists $\alpha_0 \in I_{e^*}$ so that 
$$\|\widetilde{A}^*y_{\alpha_0}^*\|>\|A\|-\delta=1-\delta.$$

 For convenience, we consider $\widetilde{A}^*y_{\alpha_0}^*$ as an element in $\Lin^N(X_1,\ldots,X_N;\mathbb{K})$ instead of $\Lin(Z;\mathbb{K})$ according to an isometric correspondence. By the assumption that $\BS(X_1,\ldots,X_N;\mathbb{K})$ is dense in $\Lin^N(X_1,\ldots,X_N;\mathbb{K})$, there exists $\varphi \in \BS(X_1,\ldots,X_N;\mathbb{K})$ with $\|\varphi\|=\|\widetilde{A}^*y_{\alpha_0}^*\|$ such that $\|\varphi-\widetilde{A}^*y_{\alpha_0}^*\|<\delta$. 

For $B \in \Lin^N(X_1,\ldots,X_N;Y)$ defined by
$$
B(\cdot) := A(\cdot) + \left[ \left( 1+\frac{\eps}{2} \right) \varphi(\cdot) - (\widetilde{A}^*y_{\alpha_0}^*)(\cdot) \right] y_{\alpha_0},
$$
we deduce that $\widetilde{B}^*$ attains its norm at $y_{\alpha_0}^*$. Indeed, we have that $\widetilde{B}^*y_{\alpha_0}^* = \left(1 +\dfrac{\eps}{2}\right) \varphi \in \BS(X_1,\ldots,X_N;\mathbb{K})$ and so 
$$
\left\|\widetilde{B}^*y_{\alpha_0}^*\right\|> \left(1 +\frac{\eps}{2}\right)(1-\delta).
$$
On the other hand, for $\alpha \in I \setminus \{\alpha_0\}$,
\begin{align*}
 \|\widetilde{B}^*y_{\alpha}^*\|
&\leq  \left\|\widetilde{A}^*y_{\alpha}^*\right\| + \left( \left\|\frac{\eps}{2} \varphi \right\|+\left\| \varphi- \widetilde{A}^*y_{\alpha_0}^* \right\| \right) \left|y_{\alpha}^*(y_{\alpha_0})\right|\\
&\leq 1 + r_{e^*} \left( \frac{\eps}{2} + \delta \right)<\left(1 +\frac{\eps}{2}\right)(1-\delta).
\end{align*}

Since we already have that $\|B-A\|<\dfrac{\eps}{2}+\delta<\eps$, it remains to show that $B\in\BS(X_1,\ldots,X_N;Y)$. To do so, it is enough to prove that $\widetilde{B}^*y_{\alpha_0}^* \perp_B \widetilde{C}^*y_{\alpha_0}^*$ for an arbitrary nonzero $C \in \Lin^N(X_1,\ldots,X_N;Y)$ with $B \perp_B C$. In other words, since $\widetilde{B}^*y_{\alpha_0}^* \in \BS(X_1,\ldots,X_N;\mathbb{K})$, there exists $(x_1,\ldots,x_N) \in M_{\widetilde{B}^*y_{\alpha_0}^*}$ such that $\widetilde{B}^*y_{\alpha_0}^*(x_1,\ldots,x_N) \perp_B \widetilde{C}^*y_{\alpha_0}^*(x_1,\ldots,x_N)$ . This gives that 
$$y_{\alpha_0}^*(B(x_1,\ldots,x_N))=\left\|\widetilde{B}^*y_{\alpha_0}^*\right\|=\|B\| \quad \text{and} \quad y_{\alpha_0}^*(C(x_1,\ldots,x_N))=0$$
 which deduce $(x_1,\ldots,x_N) \in M_B$ and $B(x_1,\ldots,x_N) \perp_B C(x_1,\ldots,x_N)$.

To prove the claim, we assume that $B \perp_B C$. From (iii) of property quasi-$\beta$, we have that
$$
\|\widetilde{B}^* + \lambda \widetilde{C}^*\| = \sup_{\alpha \in I} \left\| \widetilde{B}^*y_\alpha^* + \lambda \widetilde{C}^*y_\alpha^* \right\| \quad \text{for any } \lambda \in \mathbb{K}.
$$
Thus for $0 < |\lambda| < \dfrac{1}{\|C\|} \left[ (1-\delta) \left( 1 + \dfrac{\eps}{2} \right) - \left( 1 + r_{e^*} \left( \dfrac{\eps}{2} + \delta \right) \right) \right]$, we have

\begin{align*}
\sup_{\alpha \in I \setminus \{\alpha_0\}} \|\widetilde{B}^*y_\alpha^* + \lambda \widetilde{C}^*y_\alpha^*\| 
&\leq \sup_{\alpha \in I \setminus \{\alpha_0\}}\|\widetilde{B}^*y_\alpha^*\|+ |\lambda|\sup_{\alpha \in I \setminus \{\alpha_0\}}\| \widetilde{C}^*y_\alpha^*\| \\
&<\left( 1 + r_{e^*} \left( \frac{\eps}{2} + \delta \right) \right) +\left[ (1-\delta) \left( 1 + \frac{\eps}{2} \right) - \left( 1 + r_{e^*} \left( \frac{\eps}{2} + \delta \right) \right) \right]\\
&= (1-\delta)\left( 1 + \frac{\eps}{2} \right)\\
&< \|\widetilde{B}^*y_{\alpha_0}^*\|.
\end{align*}
Since $B \perp_B C$ implies $\widetilde{B}^* \perp_B \widetilde{C}^*$ which means that $\|\widetilde{B}^* + \lambda \widetilde{C}^*\| \geq \|\widetilde{B}^*\|$ for every $\lambda \in \mathbb{K}$, we have
$$
\|\widetilde{B}^*y_{\alpha_0}^*\| \leq \|\widetilde{B}^*y_{\alpha_0}^* + \lambda \widetilde{C}^*y_{\alpha_0}^*\| \quad \text{for } \lambda \in \mathbb{K}
$$
 by convexity of the norm, and so we finish the proof.
\end{proof}

It is shown in \cite{KL} that $\BS(X;\R)$ is dense in $\Lin(X;\R)$ whenever $X$ is locally uniformly convex. It is not difficult to see that the underlying base field can be extended to the complex plane for an analogous result, so we have the following immediate result.

\begin{cor}\label{cor:LUR-beta}
Let $X$ be a locally uniformly convex Banach space and $Y$ be a Banach space with property quasi-$\beta$. Then, $\BS(X;Y)$ is dense in $\Lin(X;Y)$.
\end{cor}

Lindenstrauss showed in \cite{L} that there is a weaker condition than property $\alpha$ but still $\NA(X;Y)$ is dense in $\Lin(X;Y)$ for such $X$. Hence, it can also be asked whether it works as property $\alpha$ is a universal condition of domain spaces for $\BS(X;Y)$ to be dense in $\Lin(X;Y)$. Recall that a subset $\{x_\alpha\}_{\alpha \in I} \subset S_X$ for some index set $I$ is said to be \emph{uniformly exposed} if there exists $\{x_\alpha^*\}_{\alpha \in I} \subset S_{X^*}$ such that $x_\alpha^*(x_\alpha)=1$ for all $\alpha \in I$ and that for any given $\eps>0$, there is $\delta>0$ so that $\|x-x_\alpha\|<\eps$ whenever $\alpha \in I$ and $x \in B_X$ satisfy $\re x_\alpha^*(x)>1-\delta$.

\begin{question}
Let $X$ be a Banach space such that $B_X$ is the absolutely closed convex hull of a uniformly exposing set. Is $\BS(X;Y)$ dense in $\Lin(X;Y)$ for every Banach space $Y$?
\end{question}

\section{Operators with the Bhatia-\v{S}emrl property on $c_0$}\label{section:c0}

We focus on the case of operators defined on the null sequence space $c_0$. In \cite{K}, the author showed that $\BS(c_0;Y)$ only consists of a zero operator whenver $Y$ is a strictly convex Banach space, and it is applied to show that $\BS(c_0;c_0)$ is not dense in $\Lin(c_0;c_0)$. We improve these results in further ways to observe that the operators with the $\BS$ property do not play well when the domain space is $c_0$. In this section, for a set $A\subset \mathbb{N}$, the notion $P_A \colon c_0 \to \ell_\infty^A \subset c_0$ denotes the canonical projection on the components in $A$. 
\begin{theorem}\label{mainprop}
Let $Y$ be any Banach space and $T \in \NA(c_0;Y)$ be given. If there is a finite set $A \subset \N$ so that $\|TP_A\|=\|T\|$ and $\|T(I-P_A)\|<\|T\|$, then $T$ does not have the $\BS$ property.
\end{theorem}
\begin{proof}
Since $TP_A$ can be considered as an operator defined on a finite dimensional space, it attains its norm at some $x_0\in S_{c_0}$ whose support belongs to $A$ which means that $x_0=P_Ax_0$.

 To prove the statement, we construct an operator $U\in \Lin(c_0;Y)$ such that $T$ is orthogonal to $U$ in the sense of Birkhoff-James but $Tz$ is not orthogonal to $Uz$ for any norm attaining point $z \in M_T$. The desired operator is defined by
\begin{displaymath}
Ue_i:=\left\{\begin{array}{@{}cl}
\displaystyle \, -Te_i & \text{if } i \in A \\
\displaystyle \, 2^{r-i}Tx_0 & \text{if } i > r \\
\displaystyle \, 0 & \text{otherwise},
\end{array} \right.
\end{displaymath}
where $r$ is the largest element in $A$ and $\{e_i\}_{i=1}^\infty$ is the canonical basis of $c_0$. We comment that this is the one that firstly considered in the proof of \cite[Theorem 2.6]{K}.

To show that $T \perp_B U$, take a functional $y_0^*\in S_{Y^*}$ so that $y_0^*(Tx_0)=\|T\|$. Since we have
$$
\left|y_0^*\left(T(x_0\pm(I-P_A)x)\right)\right|\leq \|T\|  \text{~and~} \|x_0\pm(I-P_A)x\|=1
$$
for any $x \in B_{c_0}$, we obtain
$$
y_0^*\left(T(x_0\pm(I-P_A)x)\right)=\|T\|
$$
by the equality 
$$\|T\|=y_0^*(Tx_0)=\frac{y_0^*\left(T(x_0+(I-P_A)x)\right)+y_0^*\left(T(x_0-(I-P_A)x)\right)}{2}.$$
For arbitrary $\lambda\in \mathbb{K}$ and $m>r$, we have
\begin{align*}
y_0^*\left[\left(T+\lambda  U\right)\left(x_0+\sum_{i=r+1}^m e_i\right)\right] &=y_0^*\left[T\left(x_0+(I-P_A)\left(\sum_{i=r +1 }^m e_i\right)\right)\right]+ \lambda y_0^*(Ux_0)+\lambda y_0^*\left(\sum_{i=r +1 }^m Ue_i\right)\\
&=\|T\|-\lambda\|T\|\left(1-\sum_{i=1 }^{m-r} 2^{-i}\right).
\end{align*}
By letting $m$ goes to infinity, we get $\|T+\lambda U\|\geq \|T\|$ and so  $T \perp_B  U$.

Now we claim that $Tz$ is not orthogonal to $Uz$ in the sense of Birkhoff-James for any norm attaining point $z \in M_T$. In that case, it suffices to show that
$$
\|(T+U)(I-P_A)x\|<\|T\| \quad \text{for any } x=\bigl(x(i)\bigr)_{i=1}^\infty\in B_{c_0}.
$$
If the claim is true, we have for $0<\lambda<1$ that
$$\|Tz+\lambda Uz\|=\|(1-\lambda)Tz+\lambda (T+U)(I-P_A)z\|<\|T\|=\|Tz\|$$
 which implies $Tz$ is not orthogonal to $Uz$ in the sense of Birkhoff-James.

For $\alpha=\sum_{i>r}2^{r-i}x(i)$, we have 
\begin{eqnarray*}
\|(T+U)(I-P_A)x\|
&=&\|T(I-P_A)x+U(I-P_A)x\|\\
&=&\left\|T(I-P_A)x+\alpha Tx_0\right\|\\
&=&\left\|T\left(\alpha x_0+(I-P_A)x\right)\right\|
\end{eqnarray*}
Note that $|\alpha|<1$. If $\alpha=0$, it then follows that
$$
\|(T+U)(I-P_A)x\|=\left\|T(I-P_A)x\right\|<\|T\|,
$$
so there is nothing to prove. Otherwise, $\sgn(\alpha) =\dfrac{\alpha}{|\alpha|}$ is well-defined, and so we deduce the following inequality
\begin{align*}
\|(T+U)(I-P_A)x\| &= \left\|T\left(\alpha x_0+(I-P_A)x\right)\right\|\\
&=\left\|T\left(|\alpha| x_0+\overline{\sgn(\alpha)}(I-P_A)x\right)\right\|\\
&=\left\|T\left(|\alpha|\left(x_0+\overline{\sgn(\alpha)}(I-P_A)x\right)\right)+T\left(\left(1-|\alpha|\right)\overline{\sgn(\alpha)}(I-P_A)x\right)\right\|\\
&=\left\||\alpha|T\left(x_0+\overline{\sgn(\alpha)}(I-P_A)x\right)+\left(1-|\alpha|\right)\overline{\sgn(\alpha)}T(I-P_A)x\right\|\\
&<|\alpha|\|T\|+(1-|\alpha|)\|T\|=\|T\|,
\end{align*}
as we claimed.
\end{proof}

The next lemma allows us to deal with many cases for norm attaining operators which admit the first condition of Theorem \ref{mainprop}.

\begin{lemma}\label{basic}
Let $Y$ be any Banach space and $T \in \NA(c_0;Y)$ be given. Then, there exists a finite set $A \subset \N$ such that $\|TP_A\|=\|T\|$.
\end{lemma}

\begin{proof}
Let $T$ attains its norm at $x_0=\bigl(x_0(i)\bigr)_{i=1}^\infty \in S_{c_0}$. Since $x_0(i)$ converges to $0$, we can take $m \in \N$ so that $|x_0(i)|<1/2$ for $i\geq m$. Since $\left\|x_0\pm (I-P_{\{1,\ldots,m\}})x_0\right\|\leq 1$ and
$$Tx_0=\frac{\left(Tx_0+T(I-P_{\{1,\ldots,m\}})x_0\right)+\left(Tx_0-T(I-P_{\{1,\ldots,m\}})x_0\right)}{2},$$
we have
$$
\left\|TP_{\{1,\ldots,m\}}x_0\right\|=\left\|Tx_0- T(I-P_{\{1,\ldots,m\}})x_0\right\|= \|T\|.
$$
Hence, $A=\{1,\ldots,m\}$ is the desired set. 
\end{proof}

We are now able to generalize \cite[Theorem 2.6]{K} by obtaining that the set of operators with the $\BS$ property is never dense in $\Lin(c_0;Y)$ for arbitrary range spaces $Y$. Note that in the case of norm attaining operators, $\NA(c_0;Y)$ is dense in $\Lin(c_0;Y)$ when $Y$ is (complex) uniformly convex (see \cite{A}).

\begin{cor}\label{cor:c_0-BS-property-not-dense}
Let $Y$ be any nontrivial Banach space. Then, $\BS(c_0;Y)$ is not dense in $\Lin(c_0;Y)$.
\end{cor}

\begin{proof}
Fix a nonzero operator $T \in \Lin(c_0;Y)$ whose support $B$ is finite. Then it is clear that
$$
\|TP_B\|=\|T\| \quad \text{and} \quad \|T(I-P_B)\|=0.
$$
Let $S$ be a norm attaining operator so that $\|S-T\|<\dfrac{1}{2}\|T\|$ which gives $\|S\|>\dfrac{1}{2}\|T\|$. From Lemma \ref{basic}, there exists a finite set $C\subset \mathbb{N}$ so that $\|SP_C\|=\|S\|$.

Hence, for $A=B\cup C$, we have $\left\|SP_A\right\|=\|S\|$ and that 
$$\left\|S(I-P_A)\right\|\leq \left\|S(I-P_B)\right\|\leq\left\|T(I-P_B)\right\|+\left\|(S-T)(I-P_B)\right\|<\frac{1}{2}\|T\|.$$
By Theorem \ref{mainprop}, $S$ does not have the $\BS$ property. Since $S$ is arbitrary, there is no operator with the $\BS$ property whose distance from $T$ is less than $\dfrac{1}{2}\|T\|$.
\end{proof}

One may notice that Corollary \ref{cor:c_0-BS-property-not-dense} induces the following negative result on the denseness of $N$-linear maps.

\begin{cor}
Let $X_2,\ldots,X_N$ and $Y$ be nontrivial Banach spaces. Then, $\BS(c_0,X_2,\ldots,X_N;Y)$ is not dense in $\Lin^N(c_0,X_N,\ldots,X_N;Y)$.
\end{cor}

\begin{proof}
The conclusion follows directly from the isometry $\Lin^N(c_0,X_2,\ldots,X_N;Y)=\Lin(c_0;\Lin^{N-1}(X_2,\ldots,X_N;Y))$ and Corollary \ref{cor:c_0-BS-property-not-dense}. Indeed, it is not difficult to see the fact that $A \in \BS(c_0,X_2,\ldots,X_N;Y)$ implies $T_A \in \BS(c_0;\Lin^{N-1}(X_2,\ldots,X_N;Y))$.
\end{proof}

It is remarkable that the compactness of an operator defined on $c_0$ and the $\BS$ property are mutually incompatible. Recall that a bounded linear operator $T \in \Lin(X;Y)$ is said to be \emph{compact} if $T(B_X)$ is relatively compact in $Y$, and we denote by $\K(X;Y)$ the set of all compact operators in $\Lin(X;Y)$. We leave below a sketch of its proof with an easy argument.

\begin{cor}\label{cor:compact-BS}
Let $Y$ be any Banach space. Then, $\BS(c_0;Y) \cap \K(c_0;Y) = \{ 0 \}$.
\end{cor}

\begin{proof}
As $\ell_1=c_0^*$ has the approximation property, every nonzero compact operator can be approximated by finite rank operators \cite{G}. Again, every finite rank operator can be approximated by operators with finite supports. Since we have shown that each open ball (i) whose center is an operator with a finite support and (ii) whose radius is less than half of the operator norm contains no operator with the $\BS$ property in the proof of of Corollary \ref{cor:c_0-BS-property-not-dense}, any nonzero compact operator does not belong to $\BS(c_0;Y)$.
\end{proof}

We recall the complex strict convexity to give more examples. A Banach space $X$ is said to be \emph{complex strictly convex} if for every $x \in S_X$, $\max_{\lambda \in {\mathbb{T}}} \|x+\lambda y\| =1$ implies $y=0$. It is worth to note that the strict convexity implies the complex strict convexity, and these two geometric properties are equivalent for real spaces. This is the reason why we usually consider complex convexity only for complex spaces. For a complex strictly convex Banach space $Y$, we see that every norm attaining operator $T\in \Lin(c_0;Y)$ is compact.  Indeed, by Lemma \ref{basic}, there exists an element $u\in S_{c_0}$ whose support belongs to a finite set $A \subset \N$ such that $\|Tu\|=\|T\|$. Then, it is clear that $\|T\left(u+\lambda (I-P_A)v\right)\|=\|T\|$ for arbitrary $v\in S_{c_0}$ and $\lambda\in \mathbb{T}$ by convexity of the norm. Hence, $x=Tu/\|T\|$ and $y=T(I-P_A)v/\|T\|$ gives that $y=0$ which shows that $T$ has a finite support. 

On the other hand, it is well known that every bounded linear operator from $c_0$ to $Y$ is compact if $Y$ contains no isomorphic copy of $c_0$ such as spaces having the RNP (see \cite[Theorem 6.26]{FHH}). So the preceding two remarks can be summarized by the following result.

\begin{cor}\label{cor:c0-example}
Let $Y$ be a Banach space satisfying one of the following conditions:
\begin{enumerate}
\item[\textup{(a)}] $Y$ contains no subspace isomorphic to $c_0$. In particular, $Y$ has the RNP.
\item[\textup{(b)}] $Y$ is \textup{(}complex\textup{)} strictly convex.
\end{enumerate}
Then, $\BS(c_0;Y)=\{0\}$.
\end{cor}

We improve \cite[Theorem 3.4]{K} by showing that there is no nontrivial operator with the $\BS$ property when $Y=c_0$. In this case, neither all the operators satisfy the condition stated in Theorem \ref{mainprop} nor all the norm attaining operators are compact.  For instance, we may consider the identity operator $\Id \in \Lin(c_0;c_0)$.

\begin{prop}
There is no nonzero operator $T \in \Lin(c_0;c_0)$ with the $\BS$ property. That is, $\BS(c_0;c_0) = \{0\}$.
\end{prop}

\begin{proof}
It is enough to prove that $T \in \NA(c_0;c_0)$ with $\|T\|=1$ does not have the $\BS$ property. We denote each $T_i \in \Lin(c_0;\mathbb{K})$ for $i \in \N$ to be the $n^{\text{th}}$ coordinate projection of $T$. Let
$$
\Omega := \{ i \in \N \colon \|T_i\|=1 \text{ and } T_i \in \NA(c_0;\mathbb{K}) \}.
$$

It follows easily that $\Omega$ is a nonempty set, and each $T_i$ for $i\in \Omega$ has a finite support $A_i$ since $\mathbb{K}$ is strictly convex. For $i \in \Omega$, as in the proof of Theorem \ref{mainprop}, define $R_i \in \Lin(c_0;\mathbb{K})$ by
\begin{displaymath}
R_i(e_j):=\left\{\begin{array}{@{}cl}
\displaystyle \, -T_ie_j & \text{if } j \in A_i \\
\displaystyle \, 2^{r_i-j}T_ix_i & \text{if } j > r_i \\
\displaystyle \, 0 & \text{otherwise},
\end{array} \right.
\end{displaymath}
where $x_i$ is a norm attaining point with the support $A_i$ and $r_i$ is the largest element in $A_i$. From the proof of Theorem \ref{mainprop}, we observe that
$$
T_i \perp_B R_i, \quad \|R_i\| = 2 \quad \text{and} \quad |T_i x + \lambda R_i x| < 1 \text{ for any } x \in M_{T_i} \text{ and }  0< \lambda <1.
$$

Define $S = (S_i)_i \in \Lin(c_0;c_0)$ by $S_i = \dfrac{1}{i} R_i$ if $i \in \Omega$ and $S_i =0$ otherwise. Then, $T \perp_B S$ since for $i \in \Omega$, we have that $T_i \perp R_i$, and thus 
$$
\|T\| = \|T_i\| \leq \left\|T_i + \frac{\lambda}{i} R_i\right\| \leq \|T+\lambda S\| \quad \text{for any } \lambda \in \mathbb{K}.
$$
It remains to show that $Tx_0$ is not orthogonal to $Sx_0$ in the sense of Birkhoff-James for any $x_0 \in M_T$. Define
$$
\Phi := \left\{i \in \N \colon |T_i x_0| =1 \right\}.
$$
Note that $\Phi$ is a finite subset of $\Omega$ and $\sup_{i\in \Phi^c} |T_i x_0|<1$. Moreover, we have the followings:
\begin{enumerate}
\item[\textup{(i)}] $\displaystyle \sup_{i\in \Phi}\left|T_ix_0 + \frac{\lambda}{i} R_i x_0\right| <1$ for any $0< \lambda <1$,
\item[\textup{(ii)}] $\displaystyle \sup_{i \in \Phi^c \cap \Omega}\left|T_ix_0 + \frac{\lambda}{i} R_i x_0\right| <1-|\lambda|$ for any $\lambda \in \mathbb{K} \text{ with } \displaystyle 0< |3\lambda| < 1 - \sup_{i\in \Phi^c} |T_i x_0|$,
\item[\textup{(iii)}] $\displaystyle \sup_{i \in \Omega^c}|T_ix_0 | \leq \sup_{i\in \Phi^c} |T_i x_0|<1. \phantom{\frac{\lambda}{i}}$
\end{enumerate}
Consequently, $\|Tx_0+\lambda Sx_0\| <1$ whenever $0<3\lambda < 1- \sup_{i\in \Phi^c} |T_i x_0|$, which shows that $Tx_0$ is not orthogonal to $Sx_0$ in the sense of Birkhoff-James.
\end{proof}

Finally, as a direct consequence of Corollary \ref{cor:c0-example} we are able to produce an example which distinguishes the $\BS$ property with the typical norm attainment for bilinear forms. Recall from \cite{KLM} that the set of norm attaining bilinear forms on $c_0 \times c_0$ is dense in $\Lin^2(c_0,c_0;\mathbb{K})$.

\begin{example}
There is no nonzero bilinear form with the $\BS$ property in $\Lin^2(c_0,c_0;\mathbb{K})$.
\end{example}

We finish the section with a very natural question. Let us bring to mind that $\BS(L_1[0,1];Y)= \{0\}$ for every Banach space $Y$ (see \cite{KL}). As $B_{c_0}$ does not have any extreme point as well as $B_{L_1[0,1]}$, it is natural to ask if the same kind of result can be derived. Or we can ask the denseness question for an arbitrary Banach space $Y$ when $B_X$ has no extreme point.

\begin{question}
Is it true that $\BS(c_0;Y) = \{0\}$ for an arbitrary Banach space $Y$?
\end{question}

\begin{question}
Let $X$ be a Banach space be such that $\operatorname{Ext}(B_X)=\emptyset$. Is it true that $\BS(X;Y)$ is not dense in $\Lin(X;Y)$ for every Banach space $Y$?
\end{question}


\begin{thebibliography}{99}
\bibitem{A} \textsc{M.~D.~Acosta}, The Bishop-Phelps-Bollob\'as property for operators on $C(K)$, \emph{Banach J. Math. Anal.} \textbf{10 (2)} (2016), 307--319.


\bibitem{AAP} \textsc{M.~D.~Acosta, F.~J.~Aguirre and R.~Pay\'a}, A new sufficient condition for the denseness of norm attaining operators, \emph{Rocky Mountain J. Math.} \textbf{26} (1996), 407--418.

\bibitem{AFW} \textsc{R.~M.~Aron, C.~Finet and E.~Werner}, Norm attaining $n$-linear forms and the Radon-Nikod\'ym property, \emph{$2^{\text{nd}}$ Conf. on Functions Spaces Lecture notes in Pure and Appl. Math} (1995), 19--28.

\bibitem{BS} \textsc{R.~Bhatia and  P.~\v{S}emrl}, Orthogonality of matrices and some distance problems, \emph{Linear Algebra Appl.} \textbf{287} (1999), 77--85.

\bibitem{B} \textsc{G.~Birkhoff}, Orthogonality in linear metric spaces, \emph{Duke Math. J.} \textbf{1} (1935), 169--172.

\bibitem{CS} \text{Y.~S.~Choi and H.~G.~Song}, Property (quasi-$\alpha$) and the denseness of norm attaining mappings, \emph{Math. Nachr.} \textbf{281} (2008), 1264--1272.

\bibitem{DU} \text{J.~Diestel and J.J. Uhl}, \emph{Vector Measures}, Amer. Math. Soc., Math. Surveys {\bf 15}, 1977.

\bibitem{FHH} \text{M.~Fabian, P.~Habala, P.~H\'ajeck, V. Montesinos Santaluc\'ia, J. Pelant and V. Zizler}, Functional analysis and infinite-dimensional Geometry, \emph{CMS Books in Mathematics} \textbf{8}. Springer-Verlag, New York, 2001. x+451 pp. ISBN: 0-387-95219-5

\bibitem{G} \text{A.~Grothendieck}, Produits tensoriels topologiques et espaces nucléaires, \emph{Mem. Amer. Math. Soc.} \textbf{16} (1955) 140 pp. (in French).

\bibitem{J} \text{R.~C.~James}, Orthogonality and linear functionals in normed linear spaces. \emph{Trans. Amer. Math. Soc.} \textbf{61} (1947) 265--292. 

\bibitem{K} \textsc{S.~K.~Kim}, Quantity of operators with Bhatia-\v{S}emrl property, \emph{Linear Algebra Appl.} \textbf{537} (2018), 22-37.

\bibitem{KL} \textsc{S.~K.~Kim and H.~J.~Lee}, The Birkhoff-James orthogonality of operators on infinite dimensional Banach spaces, \emph{Linear Algebra Appl.} \textbf{582} (2019), 440--451.

\bibitem{KLM} \textsc{S.~K.~Kim, H.~J.~Lee and M.~Mart\'in}, Bishop-Phelps-Bollob\'as property for bilinear forms on spaces of continuous functions, \emph{Math. Z.} \textbf{283} (2016), 157--167.

\bibitem{LS} \textsc{C.~K.~Li and H.~Schneider}, Orthogonality of matrices, \emph{Linear Algebra Appl.} \textbf{347} (2002), 115--122.

\bibitem{Lima} \textsc{A.~Lima}, Intersection properties of balls in spaces of compact operators, \emph{Ann. Inst. Fourier Grenoble} \textbf{28} (1978), 35--65.

\bibitem{L} \textsc{J.~Lindenstrauss}, On operators which attain their norm, \emph{Israel J. Math.} \textbf{1} (1963), 139--148.

\bibitem{PS} \textsc{K.~Paul and D.~Sain}, Orthogonality of operators on $(R^n,\|\,\|_\infty)$, \emph{Novi Sad J. Math.} \textbf{43} (2013), 121--129.

\bibitem{PSG} \textsc{K.~Paul, D.~Sain and P.~Ghosh}, Birkhoff-James orthogonality and smoothness of bounded linear operators, \emph{Linear Algebra Appl.} \textbf{506} (2016), 551--563.

\bibitem{SP} \textsc{D.~Sain and K.~Paul}, Operator norm attainment and inner product spaces, \emph{Linear Algebra Appl.} \textbf{439} (2013), 2448--2452.

\bibitem{SPH} \textsc{D.~Sain, K.~Paul and S.~Hait}, Operator norm attainment and Birkhoff-James orthogonality, \emph{Linear Algebra Appl.} \textbf{476} (2015), 85--97.

\bibitem{S} \textsc{C.~Stegall}, Optimization and differentiation in Banach spaces, \emph{Linear Algebra Appl.} \textbf{84} (1986), 191--211.

\bibitem{Z} \textsc{V.~Zizler}, On some extremal problems in Banach spaces, \emph{Math. Scand.} \textbf{32} (1973), 214--224.

\end{thebibliography}
\end{document}